\documentclass[11pt]{article}
\usepackage{stmaryrd}
\usepackage{amsmath}
\usepackage{amssymb,amsfonts,amsmath,amsthm}
\usepackage{epsfig}
\usepackage{pst-poly}     % From pstricks/contrib/pst-poly
\usepackage{pst-all}      % From PSTricks
\usepackage{amssymb}

\usepackage{mathrsfs}%有更多的数学符号，比如花体 $\mathscr P$

\begin{document}

\newtheorem{thm}{Theorem}%[section]
\newtheorem{cor}{Corollary}[section]
\newtheorem{lem}{Lemma}%[section]
\newtheorem{prop}{Proposition}[section]
\newtheorem{Def}{Definition}%[section]
\newtheorem{rem}{Remark}%[section]
\newtheorem{ex}{Example}[section]
\newtheorem{alg}{Algorithm}[section]
\newenvironment{pf}[1][Proof]{\noindent\textbf{#1.} }{\hfill\rule{1mm}{2mm}}

\def\theequation{\thesection.\arabic{equation}}
\makeatletter \@addtoreset{equation}{section} \makeatother

\title{\bf Conditional Fault Diagnosis of Bubble Sort Graphs under the PMC Model
\thanks {The work was supported  by NNSF of China (No. 61072080, No.11071233, 60973014),
Specialized Research Fund for the Doctoral Program of Higher
Education of China (No.200801411073), and the Key Project of Fujian
Provincial Universities services to the western coast of the
straits-Information Technology Research Based on Mathematics.} }

\author{Shuming Zhou$^{1, 3}$\quad Jian Wang$^{2}$\quad  Xirong Xu$^2$\quad Jun-Ming Xu$^3$\\
\small $^1$Key Laboratory of Network Security and Cryptology,
Fujian Normal University, \\
\small Fuzhou, Fujian, 350007, P.R.China  \\
\small $^2$Department of Computer Science, Dalian University of
Technology,\\
 \small Dalian, 116024, P.R.China\\
\small $^3$Department of Mathematics, University of Science and Technology of China, \\
\small Hefei, Anhui, 230026, P.R. China\\
}
\maketitle

{\bf\noindent Abstract:}\ As the size of a multiprocessor system
increases, processor failure is inevitable, and fault identification
in such a system is crucial for reliable computing. The fault
diagnosis is the process of identifying faulty processors in a
multiprocessor system through testing. For the practical fault
diagnosis systems, the probability that all neighboring processors
of a processor are faulty simultaneously is very small, and the
conditional diagnosability, which is a new metric for evaluating
fault tolerance of such systems, assumes that every faulty set does
not contain all neighbors of any processor in the systems. This
paper shows that the conditional diagnosability of bubble sort
graphs $B_n$ under the PMC model is $4n-11$ for $n \geq 4$, which is
about four times its ordinary diagnosability under the PMC model.

\vskip 0.3cm \noindent{\bf Keywords}:\ Conditional diagnosability;
Fault diagnosis; Bubble sort graphs.

\vskip 0.5cm

\section{Introduction}
With the rapid development of multi-processor systems, fault
diagnosis of interconnection networks has become increasingly
prominent. As a significant increase in the number of processors,
processor failure is inevitable. In order to ensure the stable
running of the systems, we must find out the faulty processors and
repair or replace them. System-level diagnosis, as a powerful tool,
has been widely used. The basic idea is to design an effective
algorithm to find out faulty processors through a comprehensive
analysis of test results which are stimulated by adjacent
processors. This method does not have to use special equipment to
test.

Most of the recent research efforts in system-level diagnosis have
focused on enhancing the applicability of system-level
diagnosis-based approaches to practical scenarios such as VLSI
testing \cite{hat98}, diagnosis of interconnection networks employed
in parallel computers \cite{ltch05,ltth04}. The classical
diagnosability of a system is small owing to the fact that it
ignores the unlikelihood of the corresponding processors failing at
the same time. Therefore, it is attractive to develop more different
measures of diagnosability based on application environment, network
topology, network reliability, and statistics related to fault
patterns. The self-diagnosis of system is implemented without
additional cost which has a very high value in practice.

The PMC model, proposed by Preparata {\it et al.} \cite{pmc67} for
dealing with the System's self-diagnosis, assumed that each node can
test its neighboring nodes, and test results are "faulty" or
"fault-free". Under this model, the diagnosability of an
interconnection network is the maximum number of faulty nodes in the
system that can be identified. To grant more accurate measurement of
diagnosability for a large-scale processing system, Lai {\it et al.}
\cite{ltch05} introduce the conditional diagnosability of a system
under the PMC model, which suppose the probability that all adjacent
nodes of one node are faulty simultaneously is very small. That is,
conditional diagnosability is the diagnoability under the condition
that all adjacent nodes of any node can't be faulty simultaneously.
They further showed that the conditional diagnosability of $Q_n$ is
$4(n-2)+1$ for $n \geq 5$. Xu {\it et al.} \cite{xth09} established
the conditional diagnosability of matching composition networks MCN
under PMC model, which generalized the result on BC networks
investigated by Zhu \cite{z08}. Xu {\it et al.} \cite{xhs10} studied
the conditional diagnosability of Shuffle-cubes under the PMC model.
Zhu {\it et al.} \cite{z08} showed that the conditional
diagnosability of folded hypercubes $FQ_n$ is $4n-3$ for $n\geq 8$.
Recently, Fan {\it et al.} ~\cite{fyzzz09} have derived the
diagnosability of $DCC$ linear congruential graphs under the precise
and pessimistic strategies based on the PMC diagnostic model. N.W.
Chang, S.Y. Hsieh ~\cite{ch10} studied the conditional
diagnosability of augmented cubes under the PMC model.

This paper establishes the conditional diagnosability of the bubble
sort graph $B_n$ under the PMC model. The remainder of this paper is
organized as follows. In Section 2, we introduce some terminology
and preliminaries used through this paper. Section 3 concentrates on
the conditional diagnosability of $B_n$. Section 4 concludes the
paper.

\section{Terminologies and Preliminaries}
For notation and terminology not defined here we follow \cite{x01}.
A multi-processor system, whose topological structure is an
interconnection network, can be modeled as a simple undirected graph
$G(V,E)$, where a vertex $u\in V$ represents a processor and an edge
$(u,v)\in E$ represents a link between vertices $u$ and $v$. If at
least one end of an edge is faulty, the edge is said to be faulty;
otherwise, the edge is said to be fault-free. The connectivity of a
graph $G$, denoted by $\kappa(G)$, is the minimum number of vertices
whose deletion results in a disconnected graph or a trivial graph.
The components of a graph $G$ are its maximal connected subgraphs. A
component is trivial if it has no edges; otherwise, it is
nontrivial. The neighborhood set of the vertex set $X\subset V(G)$
is defined as $N_G(X)=\{y\in V(G)\ |\ \exists\ x\in X\ such\ that\
(x,y)\in E(G)\}-X$. For convenience, $\lvert S\rvert$ denotes the
number of elements in the set $S$. And we also use $\lvert G\rvert$
to represent the number of vertices in the graph $G$.

The $PMC$ model requires that $u$ and $v$ can test each other for
any edge $(u,v)\in E(G)$. When $u$ tests $v$, we call $u$ as testing
node, and call $v$ as tested node. The test output is 0 (or 1) which
implies that $v$ is faulty ( or faulty-free). $\sigma(u,v)$ denotes
the output of $u$ testing $v$. And it is assumed that the test
outputs are correct if the testing node is fault-free; otherwise the
outputs are unreliable.

The collection of all outputs is called the syndrome $\sigma$. For a
given syndrome $\sigma$, a subset of vertices $F\subset V(G)$ is
said to be consistent with $\sigma$ if the syndrome $\sigma$ can be
produced from the situation that, for $\forall(u,v)\in E$ such that
$u\in V-F$, $\sigma(u,v)=1$ if and only if $v\in F$. It means that
$F$ is a possible set of faulty nodes. Since test output produced by
a faulty node is unreliable, a given set $F$ of faulty nodes may
produce different syndromes. On the other hand, different faulty
sets may produce the same syndrome. Let $\sigma (F)$ represent the
set of all syndromes that could be produced by $F$. Two distinct
sets $F_1,F_2 \subset V$ are said to be distinguishable if
$\sigma(F_1) \cap \sigma(F_2)=\varnothing$; otherwise, $F_1$ and
$F_2$ are said to be indistinguishable. We say that $(F_1, F_2)$ is
a distinguishable pair if $\sigma(F_1) \cap
\sigma(F_2)=\varnothing$; otherwise, $(F_1, F_2)$ is an
indistinguishable pair. We also use $F_1\Delta F_2=(F_1-F_2)\cup
(F_2-F_1)$ to denote the symmetric difference of $F_1$ and $F_2$.

\begin{Def}\cite{ltch05, pmc67}
A system $G$ is said to be $t-$diagnosable if, a given syndrome can
be produced by a unique faulty set, provided that the number of
faulty nodes present in the system does not exceed $t$. The largest
value of $t$, for which a given system $G$ is $t-$diagnosable, is
called the diagnosability of system $G$, denoted as $t(G)$.
\end{Def}

\begin{lem}\cite{ltch05, pmc67}
For any two distinct sets $F_1,F_2\subset V(G)$ of graph $G=(V,E)$,
$(F_1, F_2)$ is a distinguishable pair iff there exists a vertex
$u\in V(G)-(F_1\cup F_2)$ and there exists a vertex $v\in F_1\Delta
F_2 $ such that $(u,v)\in E(G)$.\hfill $\Box$
\end{lem}

So, if two sets $F_1$ and $F_2$ are indistinguishable, then there is
no edge between $F_1\Delta F_2 $ and  $V(G)-(F_1\cup F_2)$.

\begin{Def}\cite{ltch05}
A faulty set $F\subset V(G)$ is called a conditional fault-set, if
$N(v)\nsubseteq F$ for any vertex $v\in V(G)$.
\end{Def}

\begin{Def}\cite{ltch05}
A system $G$ is said to be conditionally $t-$diagnosable, if for any
two distinct conditional fault-sets $F_1,F_2 \subset V(G)$ with
$\lvert F_1 \rvert \leq t,\lvert F_2\rvert \leq t$, $(F_1, F_2)$ is
a distinguishable pair. The largest value of $t$ which makes system
$G$ is conditionally $t-$diagnosable is called the conditional
diagnosability of system $G$, denoted as $t_c(G)$.
\end{Def}

\begin{lem}\cite{ltch05}
A system $G$ is said to be $t-$diagnosable under the PMC model, if
and only if $\forall F_1,F_2\subset V(G)$, $F_1\neq F_2$ with
$\lvert F_1 \rvert \leq t,\lvert F_2 \rvert \leq t$, $(F_1,F_2)$ is
a distinguishable pair.\hfill $\Box$
\end{lem}

An equivalent way of stating the lemma is the following.

\begin{lem}\cite{ltch05}
A system $G$ is said to be $t-$diagnosable under the PMC model, if
and only if for an indistingushable pair of sets $F_1,F_2\subset
V(G)$, it implies that $\lvert F_1 \rvert > t$ or $\lvert F_2 \rvert
> t$.\hfill $\Box$
\end{lem}

\begin{lem}
Let $G(V,E)$ be a multi-processor system, and $(F_1, F_2)$ be an
indistinguishable conditional pair with $F_1\neq F_2$, then the
following two conditions hold:

(1)\quad $\lvert N(u)\cap \big (V-(F_1\cup F_2)\big)\rvert \geq 1$
for $u\in \big (V-(F_1\cup F_2)\big)$;

(2)\quad $\lvert N(v)\cap(F_1-F_2)\rvert \geq 1$ and $\lvert
N(v)\cap(F_2-F_1)\rvert \geq 1$ for $v\in F_1\Delta F_2$.\hfill
$\Box$
\end{lem}

\vskip6pt Let $(F_1, F_2)$ be an indistinguishable conditional pair,
and let $S=F_1\cap F_2$. By observation, every component of $G-S$ is
nontrivial. Moreover, we have

(1) for each component $C_1$ of $G-S$, if $C_1\cap (F_1\triangle
F_2)=\emptyset$, then $deg_{C_1}(v)\geq 1$ for $v\in V(C_1)$;

(2) for each component $C_2$ of $G-S$, if $C_2\cap (F_1\triangle
F_2)\neq \emptyset$, then $deg_{C_2}(v)\geq 2$ for $v\in V(C_2)$.

Network reliability is one of the major factors in designing the
topology of an interconnection network. The hypercubes and its
variants were the first major class of interconnection networks. The
$n$-star graph ($S_n$ for short), which proposed by Akers {\it et
al.}, is an attractive alternative to the hypercube~\cite{ak89}. The
bubble-sort graphs similar to the $n$-star graph~\cite{dt94}, which
belongs to the class of Cayley graphs, have been attractive
alternative to the hypercubes. They have some good topological
properties such as highly symmetry and recursive structure. In
particular, the $n$-dimensional bubble-sort graph $B_n$ is vertex
transitive, while it is not edge transitive~\cite{ljd93}. The
connectivity of $B_n$ is $n-1$  and the diameter is $n(n-1)/2$. It
was shown that finding a shortest path in $B_n$ can be accomplished
by using the familiar bubble-sort algorithm~\cite{ak89}. K. Kaneko,
Y. Suzuki~\cite{ks04} proposed a polynomial time algorithm for
finding disjoint paths in $B_n$. Y. Suzuki and K. Kaneko~\cite{sk06}
gave an $O(n^5)$-time algorithm that solves the node-to-set disjoint
paths problem in $B_n$. Y. Kikuchi and T. Araki~\cite{ka06} have
shown that the bubble sort graph $B_n$ is edge-bipancyclic for
$n\geq 5$, and $B_n-F$ is bipancyclic when $n\geq 4$ and $|F|\leq
n-3$, where $F$ is a subset of $E(B_n)$. T. Araki and Y.
Kikuchi~\cite{ak07} showed that the bubble-sort graph $B_n$ is
hyper-hamiltonian laceable for $n\geq 4$, and $B_n$ is still
hamiltonian laceable and strongly hamiltonian laceable if there are
at most $n-3$ faulty edges. L.M. Shih et al.~\cite{s09} showed that
the bubble-sort graph is fault tolerant maximal local connected.

Now, we introduce the bubble-sort graphs. An $n$-dimensional
bubble-sort graph is $(n-1)-$regular and symmetric. It has $n!$
nodes and $\frac{(n-1)n!}{2}$ edges while its connectivity and
diameter are $n-1$ and $\frac{n(n-1)}{2}$, respectively.

The bubble-sort graph can be equivalently defined by using adjacent
swapping operation with the following fashion (see Fig.1).

\begin{Def}(bubble-sort graph) An $n$-bubble-sort
graph $B_n$ has $n!$ nodes. Each node has a unique address, which is
a permutation of $n$ symbols $1, 2, \ldots, n$. A node that has an
address $u = u_1u_2\ldots u_n$ is adjacent to node whose address is
$u^i=u_1u_2\ldots u_{i-1}u_{i+1}u_i\ldots u_n$ with $1\leq i \leq
n-1$.
\end{Def}

\begin{center}
\scalebox{0.6}{\includegraphics*[75pt,407pt][396pt,773pt]{d:1.eps}}
\end{center}
\begin{center}
 Fig.1 The bubble sort graph $B_4$.
\end{center}
A very important property of the bubble-sort graph is its recursive
structure~\cite{ak07}. We decompose $B_n$ into $n$ subgraphs
$B_n^i$($i=1,2,\ldots,n$) such that each $B_n^i$ fixes $i$ in the
last position of the label strings which represents the vertices,
and so $B_n^i$ is isomorphic to $B_{n-1}$. Let $S_i=S\cap B_n^i$ for
$i=1,2,\ldots, n$. For $1\leq i\leq n$, the $i$th element of the
label of vertex $u$ in $B_n$ is represented by $u[i]$. An edge
$e=xy$ is called a pair-edge if $x[n]=y[n]$ and $x[n-1]=y[n-1]$.
$e=x'y'$ is called the coupled pair-edge corresponding to $e=xy$,
where $x'[n]=x[n-1]$, and $y'[n]=y[n-1]$. We call two edge, $xx'$
and $yy'$, the coupler or two pair-edge $e=xy$ and $e'=x'y'$.

Let $S$ be a faulty set of $V(B_n)$. Denote $A_1=\{ B_n^i\ |\ B_n^i$
contains at least $n-2$ nodes in $S\}$, and $A_2=\{ B_n^i\ |\ B_n^i$
contains at most $n-3$ nodes in $S\}$. we also denote $A_2$ the
subgraph of $B_n$ induced by the union of subgraphs in $A_2$.

\begin{lem}
$A_2-S$ is connected.
\end{lem}
\begin{proof}
If $|A_2|=0$ then there is nothing to do, and so assume $|A_2|\geq
1$. If $|A_2|=1$, then the lemma holds, since $B^i_n$ in $A_2$ is
$n-2$-connected. Assume $|A_2|\geq 2$ below. To prove the lemma, we
only need to show that $B_n^i$ and $B_n^j$ are connected in $A_2-S$
for any two distinct $B_n^i$ and $B_n^j$ in $A_2$.

Obviously, each of $B_n^i$ and $B_n^j$ is connected. Since vertex
$u=u_1u_2\ldots u_{n-2}ij\in B_n^j$ links to vertex $v=u_1u_2\ldots
u_{n-2}ji\in B_n^i$ with the same first $n-2$ positions of the label
strings, there are $(n-2)!$ matching edges between $B_n^i$ and
$B_n^j$, that is to say these edges are non-adjacent. Because
$(n-2)!>2(n-3)$ for $n\geq 5$, there exists at least one fault edge
between $B_n^i$ and $B_n^j$. Thus $B_n^i$ is connected to $B_n^j$,
thus $A_2-F$ is connected.
\end{proof}

\section{Conditional Diagnosability of $B_n$}
We decompose $B_n$ into $n$ subgraphs $B_n^i(i=1,2,...,n)$ such that
all vertices of $B_n^i$ have the same last bit $i$ of the label
strings which represents the vertices and $B_n^i \cong B_{n-1}$.

\begin{lem}
$t_c(B_n)\leq 4n-11$.
\end{lem}
\begin{proof}
Let $e=xy$ be one pair-edge with the coupled pair-edge $e'=x'y'$.
These two pair-edges with their coupler constitute a cycle of length
of 4. Obviously, $|N_{B_n}(x,y,x',y')|=4(n-3)$.  Let
$F_1=N\{x,y,x',y'\}\cup \{x,y\}$, $F_2=N\{x,y,x',y'\}\cup
\{x',y'\}$. It is easy to check that $F_1,F_2$ are two
indistinguishable conditional fault-sets, and $\lvert
F_1\rvert=\lvert F_2\rvert=4(n-1-2)+2=4n-10$. Thus, $t_c(B_n)\leq
4n-11$.
\end{proof}

\vskip6pt We are now ready to show the conditional diagnosability of
$B_n$ is $4n-11$ for $n \geq 5$. Let $F_1$, $F_2\subset V(B_n)$, and
$(F_1,F_2)$ be an indistinguishable conditional-pair for $n\geq 5$.
We shall show our result by proving that either $\lvert F_1
\rvert\geq 4n-10$ or $\lvert F_2 \rvert \geq 4n-10$.

\begin{lem}
For any two indistinguishable conditional fault-sets $F_1,F_2$ in
$B_n$ with $n \geq 5$, which satisfies $F_1\neq F_2$, we have either
$\lvert F_1 \rvert \geq 4n-10$ or $\lvert F_2 \rvert \geq 4n-10$.
\end{lem}

\begin{proof}
Since $(F_1, F_2)$ is an indistinguishable conditional-pair, there
exists no edge between $F_1\Delta F_2$ and $B_n-(F_1\cup F_2)$ by
Lemma 1. Define $S=F_1\cap F_2$ and let $B_n[F_1\Delta F_2]$ be the
subgraph of $B_n$ induced by the vertex set $F_1\Delta F_2$. We
choose a maximal component $C$ in $B_n[F_1\Delta F_2]$ when
$B_n[F_1\Delta F_2]$ is not connected; otherwise, let
$C=B_n[F_1\Delta F_2]$. By lemma 4, we have $|C|\geq 4$. Thus, we
only need to prove $\lceil \frac{\lvert C\rvert}{2}\rceil+\lvert
S\rvert \geq 4n-10$, which implies that $\lvert F_1\rvert \geq
4n-10$ or $\lvert F_2\rvert \geq 4n-10$.

We decompose $B_n$ into $n$ subgraphs $B_n^i$($i=1,2,\ldots,n$) such
that each $B_n^i$ fixes $i$ in the last position of the label
strings which represents the vertices. Let $S_i=S\cap B_n^i$ for
$i=1,2,\ldots, n$.

If $\lvert S\rvert \geq 4n-12$, then we have $\lceil \frac{\lvert
C\rvert}{2}\rceil+\lvert S\rvert \geq 4n-10$ for $\lvert C\rvert
\geq 4$, so the lemma holds. Now, we only consider the situation
$\lvert S\rvert \leq 4n-13$.

Let $A_1=\{ B_n^i\ |\ B_n^i$ contains at least $n-2$ nodes in $S\}$,
and  $A_2=\{ B_n^i\ |\ B_n^i$ contains at most $n-3$ nodes in $S\}$.
Obviously, $A_1$ has at most three elements by the fact that
$4(n-2)>4n-12$.

If $A_1=\phi$, by lemma 5 we have $|C|\geq (n-3)[(n-1)!-(n-3)]>
2(4n-12)$ for $n\geq 5$. Thus, we have $\lceil \frac{\lvert
C\rvert}{2}\rceil+\lvert S\rvert \geq 4n-10$. Now, we consider
$A_1\neq\phi$ as follows.

{\bf Case 1. $C\cap (A_2-S)\neq \phi$.}

Since $A_1$ has at most three subgraphs, we have $|C|\geq
(n-3)[(n-1)!-(n-3)]\geq 2(4n-10)$ for $n\geq 5$; and we arrive at
the result.

{\bf Case 2. $C\cap (A_2-S)= \phi$.}

Obviously, $C\subset A_1$, $N_{A_2}(C)\subset S\cap A_2$ by the
maximality of $C$. Now, we divide this case into three subcases
below.

{\bf Subcase 2.1. There is exactly one subgraph, say $X$, in $A_1$.}

Since every vertex of $C\cap X$ has exactly one neighbor outside of
$X$, we have $N_{B_n-X}(C)\subset S\cap A_2$; and so
$|N_{B_n-X}(C)|=|C|$. Obviously, $N_X(C)\subset S\cap X$. Since
$|C|\geq 4$, let $T$ be a path of length three in $C$. Obviously,
$N_X(C)\supseteq N_X(T)-(C-T)$, and so
$$|N_X(C)|\geq |N_X(T)|-(|C-T|)\geq 4n-12-|C|.$$

Since $S=(S\cap X)\cup (S\cap A_2)$, we have $$|S|=|S\cap X|+ |S\cap
A_2|\geq |N_X(C)|+|N_{B_n-X}(C)|\geq 4n-12.$$ Thus, we have $\lceil
\frac{\lvert C\rvert}{2}\rceil+\lvert S\rvert \geq 4n-10.$

{\bf Subcase 2.2. There are exactly two subgraphs, say $X$ and $Y$,
in $A_1$.}

We assume, without loss of generality, that $|C\cap X|\geq |C\cap
Y|$. For any $x\in C\cap X$, $|N_C(x)|\geq 2$ by Lemma 4 and every
vertex of $B_n^i$ has exactly one neighbor outside of this subgraph,
then we have $|C\cap X|\geq |C\cap Y|\geq 2$. Due to $N_X(C\cap
X)\subset S\cap X$, and $C\cap X$ has at least two vertices, $|S\cap
X|\geq |N_X(C\cap X)|\geq 2(n-2)-2$. Similarly, we have $|S\cap
Y|\geq |N_Y(C\cap Y)|\geq 2(n-2)-2$. Since every pair of $S\cap X$,
$S\cap Y$, $S\cap A_2$ are disjoint, we have $$S=(S\cap X)\cup
(S\cap Y)\cup (S\cap A_2)\supseteq N_X(C\cap X)\cup N_Y(C\cap Y)\cup
S\cap A_2.$$ Then we have
$$|S|\geq |N_X(C\cap X)|+ |N_Y(C\cap Y)|+ |S\cap A_2|\geq
2(n-2)-2+2(n-2)-2+|S\cap A_2|.$$

Thus, we have $\lceil \frac{\lvert C\rvert}{2}\rceil+\lvert S\rvert
\geq 4n-10.$

{\bf Subcase 2.3. There are exactly three subgraphs, say $X$, $Y$,
$Z$, in $A_1$.}

We assume, without loss of generality, that $|C\cap X|\geq 2$ by the
fact that $|C|\geq 4$. We have $|F\cap X|\geq |N_X(C\cap X)|\geq
2(n-2)-2$. Since every vertex of $B_n^i$ has exactly one neighbor
outside of this subgraph, and every pair of $S\cap X$, $S\cap Y$,
$S\cap Z$, $S\cap A_2$ are disjoint, we have
$$S=(S\cap X)\cup (S\cap Y)\cup (S\cap Z)\cup (S\cap A_2)\supseteq N_X(C\cap
X)\cup (S\cap Y)\cup (S\cap Z)\cup S\cap A_2.$$ Then we have
$$|S|\geq |N_X(C\cap X)|+ |S\cap Y|+|S\cap Z|+ |S\cap A_2|\geq
2(n-2)-2+(n-2)+(n-2)+|S\cap A_2|\geq 4n-10.$$

Thus, we have $\lceil \frac{\lvert C\rvert}{2}\rceil+\lvert S\rvert
\geq 4n-10.$
\end{proof}

By Lemma 6 and 7, we have
\begin{thm}
The conditional diagnosability of bubble sort graph $B_n$ under the
PMC model is $t_c(B_n)= 4n-11$ $(n\geq 5)$. \hfill $\Box$
\end{thm}

\begin{thm}
The conditional diagnosability of bubble sort graph $B_4$ under the
PMC model is $t_c(B_4)=5$.
\end{thm}
\begin{proof}
Let $F_1,F_2$ be two distinguishable fault-sets in $B_4$. Denote
$S=F_1\cap F_2$ and $C\subset F_1\Delta F_2$ be a connected
component in $B_4-S$. By Lemma 4, we have $\lvert C\rvert \geq 4$.

If $\lvert S\rvert \geq 4$, then $\lceil \frac{\lvert
C\rvert}{2}\rceil+\lceil S\rceil \geq 6$. Now we suppose that
$\lvert S\rvert \leq 3$. In the worst case, $B_4-S$ has two
components, one of which is an isolated vertex $\{u\}$, then we have
$C=B_4-S-\{v\}$, which implies $\lvert C\rvert=20$. So, $\lceil
\frac{\lvert C\rvert}{2}\rceil+\lceil S\rceil > 10$. Thus
$t_c(B_4)\geq 5$, while $t_c(B_4)\leq 5$, hence $t_c(B_4)=5$.
\end{proof}

\section{Conclusion}
The issue of identifying faulty processors is important for the
design of multiprocessor interconnected systems, which are
implementable with VLSI. The process of identifying all the faulty
processors is the system-level diagnosis. This paper establishes the
conditional fault diagnosability of the bubble sort graphs under the
PMC model. The conditional diagnosability is about four times the
traditional diagnosability under PMC model. This method can be also
applied to other complex network structure.

\end{document}